\documentclass[11pt,twoside]{amsart}

\numberwithin{equation}{section}
\usepackage{latexsym,amssymb,amsmath,amsthm,amsfonts}
\usepackage{graphicx}
\usepackage{tikz}
\usetikzlibrary{matrix,arrows}
\definecolor{VerdeOlivo}{rgb}{0.3,0.5,0.1}
\definecolor{Magenta}{rgb}{.65,0.15,.2}
\definecolor{Gris}{gray}{0.3}

\newtheorem{Theorem}{Theorem}[section] 
 
\newtheorem{Lemma}[Theorem]{Lemma} 
\newtheorem{Corollary}[Theorem]{Corollary} 
 
\newtheorem{Remark}[Theorem]{Remark}

\newtheorem{Claim}[Theorem]{Claim} 
 
\begin{document} 


\title[On the critical group of matrices]{On the critical group of matrices}


\author{Hugo Corrales}
\address{
Departamento de
Matem\'aticas\\
Centro de Investigaci\'on y de Estudios Avanzados del
IPN\\
Apartado Postal
14--740 \\
07000 Mexico City, D.F. 
} 
\email[H. ~Corrales]{hhcorrales@gmail.com}
\thanks{The first author was partially supported by CONACyT and the second author was partially supported by SNI}

\author{Carlos E. Valencia}
\email[C. ~Valencia\footnote{Corresponding author}]{cvalencia@math.cinvestav.edu.mx, cvalencia75@gmail.com}



\begin{abstract} 
Given a graph $\mathcal{G}$ with a distinguished vertex $s$, the critical group of  $(\mathcal{G},s)$ is the cokernel of their reduced Laplacian matrix $L(G,s)$.
In this article we generalize the concept of the critical group to the cokernel of any matrix with entries in a commutative ring with identity.
In this article we find diagonal matrices that are equivalent to some matrices that generalize the reduced Laplacian matrix of the path, the cycle, and the complete graph 
over an arbitrary commutative ring with identity. 
We are mainly interested in those cases when the base ring is the ring of integers and some subrings of matrices.
Using these equivalent diagonal matrices we calculate the critical group of the $m$-cones of the $l$-duplications of the path, the cycle, and the complete graph. 
Also, as byproduct, we calculate the critical group of another matrices, as the $m$-cones of the $l$-duplication of the bipartite complete graph with $m$ vertices in each partition, 
the bipartite complete graph with $2m$ vertices minus a matching. 
\end{abstract}

\maketitle


\section{Introduction}
Let $\mathcal{G}=(V,E)$ be a finite connected graph without loops, but with multiple edges allowed.
The {\it adjacency matrix} of $\mathcal{G}$, denoted by $A(\mathcal{G})$, is given by $A(\mathcal{G})_{u,v}=m_{u,v}$,
where $m_{u,v}$ is the number of the edges between $u$ and $v$ in $V$.
The {\it Laplacian matrix} of $\mathcal{G}$ is the matrix $L(\mathcal{G})=D(\mathcal{G})-A(\mathcal{G})$ where 
\[
D(\mathcal{G})_{u,v}=
\begin{cases}
d_{\mathcal{G}}(u) & \text{ if } u=v,\\
0 & \text{ otherwise},
\end{cases}
\]
and $d_{\mathcal{G}}(u)$ is the degree of the vertex $u$ in $\mathcal{G}$.
If $s$ is a vertex of $\mathcal{G}$, the {\it reduced Laplacian matrix}, denoted by $L(\mathcal{G},s)$, is the matrix obtained
from $L(\mathcal{G})$ by removing the row and column $s$.
The {\it critical group} of $\mathcal{G}$, denoted by $K(\mathcal{G})$, is the cokernel of $L(\mathcal{G},s)$,
\[
K(\mathcal{G})=\mathbb{Z}^{\widetilde{V}}/{\rm Im}\, L(\mathcal{G},s)^t,
\]
where $\widetilde{V}=V\setminus s$.

The critical group is an abelian group that is isomorphic to the sandpile group introduced by 
Dhar in~\cite{dhar90}, which generalizes the case of a grid from~\cite{bak87}.
The critical group has been studied by several authors, 
see for instance~\cite{Bai,berget09,biggs99,p4cn,mobius,threshold,k3cn,cartesian,kmpn,musiker,shen,directed,wang09,wang09p}.

The concept of critical group can be generalized easily to an arbitrary commutative ring with identity  $\mathbb{A}$.
More precisely, if $M\in M_{m\times n}(\mathbb{A})$, then the {\it critical module} of $M$, denoted by $K(M)$, is defined as:
\[
K(M):=\mathbb{A}^n/M^t\mathbb{A}^m.
\]

Given $H<GL_n(\mathbb{A})$ and $H'<GL_m(\mathbb{A})$, we say that $M,N \in M_{m\times n}(\mathbb{A})$ 
are {\it $(H,H')$-equivalent}, denoted by $N\sim_{(H,H')}M$, if there exist $P\in H$ and $Q\in H'$ such that $N=PMQ$.
When $H=SL_n(\mathbb{A})$ and $H'=SL_m(\mathbb{A})$, then we simply say that $M$ and $N$ 
are {\it unitary equivalent} and will be denoted by $N\sim_u M$.
Also, if $H$ and $H'$ are the subgroups generated by the elementary matrices, we simply say that $M$ and $N$
are {\it elementary equivalent} and will be denoted by $N\sim_e M$.
Finally, if $H=GL_n(\mathbb{A})$ and $H=GL_n(\mathbb{A})$, then we simply say that $M$ and $N$ are equivalent
and will be denoted by $N\sim_{\mathbb{A}} M$ or $M\sim N$ if the ring $\mathbb{A}$ is clear from the context.

Usually the critical group of a graph can be described in terms of a diagonal matrix called the Smith Normal Form of $L(\mathcal{G})$.
Moreover, is not difficult to see that if $M$ and $N$ are equivalent, then 
\[
K(M)=\mathbb{A}^n/M^t\mathbb{A}^m\cong \mathbb{A}^n/N^t\mathbb{A}^m=K(N).
\]
When the base ring $\mathbb{A}$ is Principal Ideal Domain (PID), another description of the critical group of a matrix $M$ is given by
\[
K(M)=\bigoplus_{i=1}^{|V|} \mathbb{A}_{\Delta_i(M)/\Delta_{i-1}(M)},
\]
where $\Delta_i(M)$ is the greatest common divisor of all the $i\times i$ minors of $M$.

This article is divided in two sections: In the first section, we find diagonal matrices over an arbitrary commutative ring with identity 
that are equivalent to some matrices that generalize the Laplacian matrices of the path, the cycle, and the complete graph.
In the second section we apply  the results obtained in the first section 
in the case when the base ring is the ring of integers and some subrings of matrices.
In particular we are able to calculate the critical group of the $m$-cones of the $l$-duplications of the path, the cycle, the complete graph,
the bipartite complete graph with $m$ vertices in each partition, the bipartite complete graph with $2m$ vertices minus a matching, etc. 


In the following, every multigraph will be connected and will have a distinguished vertex $s_G\in V(G)$, called the \textit{sink} of $G$.
Sometimes we will simply write $s$ instead of $s_G$.
In~\cite{diestel} it can be seen any unexplained term of graph theory.




\section{The critical module of matrices}\label{sec1}

In this section we will find diagonal matrices that are equivalent to some matrices that generalize the Laplacian matrices
of the path, the cycle, and the complete graphs.
After that, we will apply these results in order to calculate the critical group for several families of graphs.

For all $n\geq 2$ and $a,b \in \mathbb{A}$, let $K_n(a,b)=(a+b)I_n+bA(\mathcal{K}_n)$, $T_n(a,b)=aI_n+bA(\mathcal{P}_n)$, 
$C_n(a,b)=aI_n+bA(\mathcal{C}_n)$,
and
\[
P_n(a,b)=\left(\begin{array}{ccccc}
a\!+\!b&b&0&\ldots&0\\
b&a&\ddots&\ddots&\vdots\\
0&\ddots&\ddots&\ddots&0\\
\vdots&\ddots&\ddots&a&b\\
0&\cdots&0&b&a\!+\!b
\end{array}\right).
\]
where $I_n\in M_{n\times n}(\mathbb{A})$ is the identity matrix on order $n$. 

Since the critical module of a matrix is invariant under equivalency classes, 
then in order to determine the critical module of a matrix, it is enough to find an equivalent diagonal matrix.

\begin{Theorem}\label{prin}
Let $a,b \in \mathbb{A}$ such that the equation $ax+by=1$ has solution in $\mathbb{A}$ 
and $f_n(x,y)$ polynomials in $\mathbb{A}[x,y]$ that satisfy the recurrence relation 
\[
f_n(x,y)=xf_{n-1}(x,y)-y^2f_{n-2}(x,y)
\]
with initial values $f_{-1}(x,y)=0$ and $f_0(x,y)=1$.
Then 
\begin{description}
\item[(i)] $T_n(a,b)\sim_u {\rm diag}(1,\ldots,1,f_n(a,b))$ for all $n\geq 2$,

\item[(ii)]  $P_n(a,b)\sim_u {\rm diag}(1,\ldots,1,(a+2b)f_{n-1}(a,b))$ for all $n\geq 2$,

\item[(iii)] $K_n(a,b)\sim_u \mathrm{diag}(1,a,\ldots,a,a(a+nb))$ for all $n\geq 2$, and

\item[(iv)] $C_n(a,b)\sim_{u} I_{n-2} \oplus C$ for all $n\geq 4$, where
\[
C=
\begin{cases}
f_q(a,b)\left(\begin{array}{cc}
a&2b\\
2b&a
\end{array}\right) & \text{ if } n-2=2q,\\
\big(f_{q+1}(a,b)-bf_q(a,b)\big)\left(\begin{array}{cc}
1&0\\
0&a+2b
\end{array}\right) & \text{ if } n-2=2q+1.
\end{cases}
\]
\end{description}
\end{Theorem}
\begin{proof}
$(i)$
For all $l\geq 2$ and $1\leq k\leq l-1$, let 
\[
Z_{k,l}(a,b)
=\left(\begin{array}{ccc}
f_k&bf_{k-1} & {\bf 0}_{l-2} \\
b&\\
{\bf 0} & &T_{l-1}(a,b)
\end{array}\right)
\in M_{l}(\mathbb{A}),
\text{ and }
Z_{n,1}(a,b)=(f_n),
\]
where $f_k:=f_k(a,b)$ for all $k\geq -1$ and ${\bf 0}$ 
is the matrix with all the entries equal to $0$.
Note that $Z_{1,n}(a,b)=T_n(a,b)$ 

Now we will prove the following statement:

\begin{Claim}\label{reduction1}
For all $l\geq 2$ and $1\leq k\leq l-1$
\[
Z_{k,l}(a,b)\sim_u I_{1}\oplus Z_{k+1,l-1}(a,b).
\]
\end{Claim}
\begin{proof}
Let $x_1,y_1\in \mathbb{A}$ be a solution of the equation $ax+by=1$.
Moreover, for all $k\geq 1$, let 
\[
x_k=x_1^k \text{ and }y_k=\sum_{i=1}^k \binom{k}{i} a^ib^{k-1-i}x_1^iy_1^{k-i}-x_1^k(f_k-a^k)/b,
\]
that is, $x_k$ and $y_k$ are a solution of the equation $f_kx_k+by_k=1$.

Since
\[
\left(\begin{array}{ccc}
x_k&y_k& {\bf 0}\\
-b&f_k& {\bf 0}\\
{\bf 0} & {\bf 0} & I_{n-k-1}
\end{array}\right)
Z_{k,l}(a,b)
=
\left(\begin{array}{cccc}
1&*&* &*\\
0&af_k-b^2f_{k-1}&bf_{k} & {\bf 0}\\
0&b& \\
{\bf 0} &{\bf 0} && T_{n-k-1}(a,b)
\end{array}\right)
\]
and 
${\rm det}
\left(\begin{array}{cc}
x_k&y_k\\
-b&f_k
\end{array}\right)=f_kx_k+by_k=1$, then $Z_{k,l}(a,b)\sim_u I_{1}\oplus Z_{k+1,l-1}(a,b)$.
\end{proof}

Applying claim~\ref{reduction1}, we get that  $T_{n}(a,b)=Z_{1,n}(a,b)\sim_u I_{n-1}\oplus Z_{n,1}(a,b)$.

$(ii)$
For all $n\geq0$, $m\geq 1$, let 
\[
Z_{k,l}(a,b)=
\left(\begin{array}{ccccc}
f_k+bf_{k-1}&b(f_{k-1}+bf_{k-2})&0&\cdots&0\\
b& a&b&\cdots&0\\
0&b& & & 0 \\
0&0& &T_{l-3}(a,b)  & {\bf 0}\\
\vdots&\vdots& && b\\
0&0 & {\bf 0} & b & a+b
\end{array}\right)
\in M_{l}(\mathbb{A}).
\]
Also, let 
\[
x'_k=x_{k}\text{ and }y'_k=y_k-f_{k-1}x_k,
\] 
that is, $x_k$ and $y_k$ are a solution of the equation $(f_k+bf_{k-1})x'_k+by'_k=1$.
Since
\[
\left(\begin{array}{ccc}
x'_k&y'_k&{\bf 0}\\
-b&f_k+bf_{k-1}&{\bf 0}\\
{\bf 0}&{\bf 0}& I_m
\end{array}\right)
Z_{k,l}(a,b)
=
\left(\begin{array}{ccccc}
1&*&*&\cdots&*\\
0&f_{k+1}+bf_{k}&b(f_k+bf_{k-1})&{\bf 0}&0\\
0&b && & 0\\
0&0 & &T_{l-4}(a,b)  & {\bf 0}\\
\vdots & \vdots & & &b\\
0&0 & {\bf 0}& b& a+b 
\end{array}\right)
\]
for all $l\geq 2$
and 
${\rm det}
\left(\begin{array}{cc}
x'_k&y'_k\\
-b&f_k+bf_{k-1}
\end{array}\right)=
(f_k+bf_{k-1})x_k+by_k=1$,
then $Z_{k,l}(a,b) \sim_u I_1\oplus Z_{k+1,l-1}(a,b)$ for all $l\geq 2$.
Thus
\[
P_{n}(a,b)=Z_{1,n}(a,b)\sim_u 
I_{n} \oplus 
\left(\begin{array}{cc}
f_{n-1}+bf_{n-2}&b(f_{n-2}+bf_{n-3})\\
b&a+b
\end{array}\right).
\]
Therefore, $P_{n+2}(a,b)\sim_u  \mathrm{diag}(1,\ldots,1,(a+2b)f_{n-1}(a,b))$
because
\begin{eqnarray*}
\left(\begin{array}{cc}
x'_{n-1}&y'_{n-1}\\
-b&f_{n-1}+bf_{n-2}
\end{array}\right)
\left(\begin{array}{cc}
f_{n-1}+bf_{n-2}&b(f_{n-2}+bf_{n-3})\\
b&a+b
\end{array}\right)
&=&
\left(\begin{array}{cc}
1&*\\
0&(a+2b)f_{n-1}
\end{array}\right)\\
&\sim_e&
\left(\begin{array}{cc}
1&0\\
0&(a+2b)f_{n-1}(a,b)
\end{array}\right).
\end{eqnarray*}


$(iii)$
We begin proving the following statement:
\begin{Claim}\label{claimK}
If $n\geq 2$, then 
$
K_n(a,b)
\sim_e
\left(\begin{array}{cc}
a&b\\
na &-a
\end{array}\right)
\oplus aI_{n-2}
$.
\end{Claim}
\begin{proof}
It turns out because
\begin{eqnarray*}
P_nK_n(a,b)Q_n
&=&
\left(\begin{array}{cccccc}
b&b&\cdots&b&b&a+b\\
a&a&a&\cdots&a&-(n-1)a\\
0&a&-a&0 &\cdots&0\\
0&0&a&-a&\cdots&\vdots\\
\vdots&\vdots&\ddots&\ddots&\ddots&0\\
0&0&\cdots&0&a&-a
\end{array}\right)
Q_n\\
&=&
\left(\begin{array}{cc}
a&b\\
-na &a
\end{array}\right)
\oplus aI_{n-2},
\end{eqnarray*}
where
\[
P_n
=
\left(\begin{array}{cccccc}
0& 0&\cdots&\cdots&0&1\\
1&1&\cdots&1&1&-(n-1)\\
0& 1&-1&0&\cdots&0\\
0&0&\ddots&\ddots&\ddots&\vdots\\
\vdots&\vdots&\ddots&\ddots&\ddots&0\\
0&0&\cdots&0&1&-1
\end{array}\right)
\text{ and }
Q_n
=
\left(\begin{array}{cccccc}
-n+1&1&-1&-2&\cdots&-n+2\\
1&0&1&1&\cdots&1\\
1&0&0&1&\ddots& \vdots\\
1&0&0&0&1& \vdots\\
\vdots&\vdots&\vdots&\ddots&\ddots&1\\
1&0&0&\cdots&0&0
\end{array}\right)
\]
are elementary matrices. 
\end{proof}

On the other hand, $K_{n}(a,b)\sim_u  \mathrm{diag}(1,a,\ldots,a,a(a+nb))$
because
\[
\left(\begin{array}{cc}
a&b\\
-na&a
\end{array}\right)
\left(\begin{array}{cc}
x_1&-b\\
y_1&a
\end{array}\right)=\left(\begin{array}{cc}
1&0\\
*&a(a+nb)
\end{array}\right)
\sim_e
\left(\begin{array}{cc}
1&0\\
0&a(a+nb)
\end{array}\right)
\]
and
$
{\rm det}
\left(\begin{array}{cc}
x_1&-b\\
y_1&a
\end{array}\right)=
ax_1+by_1=1
$.

$(iv)$
For all $l\geq 4$ and $k\geq 1$, let
\[
Z_{k,l}(a,b)
:=\left(\begin{array}{ccccc}
f_k&bf_{k-1} & {\bf 0}_{l-4}&b^2f_{k-2}&bf_{k-1}\\
b& && &{\bf 0}\\
 & &T_{l-2}(a,b)& \\
{\bf 0}&&&&b\\
bf_{k-1}& b^2f_{k-2}&{\bf 0}_{l-4}&bf_{k-1} &f_k
\end{array}\right)
\in M_{l}(\mathbb{A})
\]
and
\[
Z'_{k,l}(a,b)
:=\left(\begin{array}{ccccc}
f_{k+1}&bf_{k} & {\bf 0}_{l-4}&-b^3f_{k-2}&-b^2f_{k-1}\\
b& && &{\bf 0}\\
 & &T_{l-2}(a,b)& \\
{\bf 0}&&&&b\\
-f_{k}& -bf_{k-1}&{\bf 0}_{l-4}&bf_{k-1} &f_k
\end{array}\right)
\in M_{l}(\mathbb{A}).
\]
Also, let
\[
P_{k,l}(a,b):=
\left(\begin{array}{ccc}
x_k&y_k&{\bf 0}\\
-b & f_k & {\bf 0}\\
{\bf 0}&{\bf 0}&\\
0&-f_{k-1}&I_{l-2}
\end{array}\right)
\text{ and }
Q_{k,l}(a,b):=
I_1 \oplus
\left(\begin{array}{ccc}
I_{l-3} &bf_{k-1}&0\\
& {\bf 0} & {\bf 0}\\
{\bf 0} & f_k & -b\\
{\bf 0} & y_k & x_k
\end{array}\right).
\]
Since
\[
P_{k,l}(a,b) \cdot Z_{k,l}(a,b)=
{\small
\left(\begin{array}{cccccc}
1&*&*&&*&*\\
0 & f_{k+1} & bf_k & {\bf 0}_{l-5}& -b^3f_{k-2}&-b^2f_{k-1}\\
0& b &&& &{\bf 0}\\
{\bf 0}&  &&T_{l-3}(a,b)& &\\
{\bf 0}&{\bf 0}& && &b\\
0&-f_{k}&-bf_{k-1}&{\bf 0}_{l-5}& bf_{k-1}&f_{k}
\end{array}\right)}
\sim_u
I_1 \oplus Z'_{k,l-1}(a,b)
\]
for all $l\geq 5$,
\[
Q_{k,l}(a,b)\cdot (I_1 \oplus Z'_{k,l-1}(a,b))
=
I_1 \oplus
\left(\begin{array}{cccccc}
f_{k+1} & bf_k & {\bf 0}_{l-6}& b^2f_{k-1}&bf_{k} & 0\\
b &&& &{\bf 0} &\\
&&T_{l-4}(a,b)& &&\\
{\bf 0} & && &b&\\
bf_{k}&b^2f_{k-1}&{\bf 0}_{l-6}& bf_{k}&f_{k+1}& 0\\
* & * & &* & * & 1
\end{array}\right)
\]
for all $l-1\geq 5$, and ${\rm det}(P_{k,l}(a,b))={\rm det}(Q_{k,l}(a,b))=1$,
then 
\[
Z_{k,l}(a,b)\sim_u I_1 \oplus Z_{k+1,l-2}(a,b) \oplus I_1 \text{ for all } l\geq 6.
\]

Moreover, since $Z_{1,n}(a,b)=C_n(a,b)$, then
\[
C_n(a,b)\sim_u
\begin{cases}
Z_{q,4}(a,b) & \text{ if } n-2=2q,\\
Z'_{q,4}(a,b) & \text{ if } n-2=2q+1.
\end{cases}
\]

Finally, using similar reductions we get that
\begin{eqnarray*}
Z_{q,4}(a,b)
&\sim_u& 
\left(\begin{array}{cccc}
f_q&bf_{q-1}&b^2f_{q-2}&bf_{q-1}\\
b&a&b&0\\
0&b&a&b\\
0&-f_q&0&f_q
\end{array}\right)
\sim_u 
\left(\begin{array}{cccc}
1&0&0&0\\
0&f_{q+1}&bf_q-b^3f_{q-2}&-b^2f_{q-1}\\
0&b&a&b\\
0&-f_q&0&f_q
\end{array}\right)\\
&\sim_u&
\left(\begin{array}{cccc}
1&0&0&0\\
0&af_q&2bf_q&0\\
0&b&a&b\\
0&-f_q&0&f_q
\end{array}\right)
\sim_u
\left(\begin{array}{cccc}
1&0&0&0\\
0&af_q&2bf_q&0\\
0&2bf_q&af_q&0\\
0&0&0&1
\end{array}\right)
\sim_u 
I_2\oplus f_q\left(\begin{array}{cc}
a&2b\\
2b&a
\end{array}\right)
\end{eqnarray*}
and
\begin{eqnarray*}
Z'_{q,4}(a,b)
&\sim_u&
{\small 
\left(\begin{array}{cccc}
f_{q+1}&bf_q+b^2f_{q-1}&abf_{q-1}-b^3f_{q-2}&0\\
b&a&b&0\\
0&b&a&b\\
-f_q&-bf_{q-1}&bf_{q-1}&f_q
\end{array}\right)}
\sim_u
\left(\begin{array}{cccc}
f_{q+1}&bf_q+b^2f_{q-1}&bf_q& 0\\
b&a&b&0\\
bf_q& bf_q+b^2f_{q-1}& f_{q+1}&0\\
0&0&0&1
\end{array}\right)\\
&\sim_u&
{\small
\left(\begin{array}{cccc}
1&0&0&0\\
0&af_{q+1}-b^2f_q-b^3f_{q-1}& bf_{q+1}-b^2f_q&0\\
0&bf_q+b^2f_{q-1}-af_q&f_{q+1}-bf_{q}&0\\
0&0& 0&1
\end{array}\right)}
\sim_u 
I_2\oplus \big(f_{q+1}-bf_q\big)
\left(\begin{array}{cc}
a+b&b\\
-1&1\\
\end{array}\right)\\
&\sim_u& 
I_2\oplus \big(f_{q+1}-bf_q\big)
\left(\begin{array}{cc}
1&0\\
0&a+2b
\end{array}\right).
\end{eqnarray*}
\end{proof}

\begin{Remark}
In \cite[pag. 44]{corralesthesis}, a simpler proof of theorem~\ref{prin} (iii) when $\mathbb{A}$ is a principal ideal domain can be found . 
\end{Remark}

The next Lemma give us some useful properties of the polynomials $f_n(x,y)$. 
\begin{Lemma}
If $n\geq 1$, then
\begin{description}
\item[(i)] $f_n(x,y)=\sum_{i=0}^{\lfloor \frac{n}{2}\rfloor} (-1)^i\binom{n-i}{i} x^{n-2i}y^{2i}$,
\item[(ii)] $f_n(x+y,-1)-f_n(y,-1)=x\sum_{i=0}^{n-1} f_i(x+y,-1)f_{n-i}(y,-1)$,
\item[(iii)] $f_n(x,y)=f_k(x,y)f_{n-k}(x,y)-y^2f_{k-1}(x,y)f_{n-k-1}(x,y)$,
\item[(iv)] $x^kf_n(x,y)=\sum_{i=0}^k \binom{k}{i} x^{2i}f_{n+k-2i}(x,y)$.
\end{description}
\end{Lemma}
\begin{proof}
It follows using induction on $n$. \end{proof}

\begin{Remark}
Note that, $f_n(x,0)=x^n$.
\end{Remark}

If $\mathbb{A}$ is a principal ideal domain and $a,b\in \mathbb{A}$, then the equation $ax+by=1$ has a solution if and only if ${\rm gcd}(a,b)=1$.  

\begin{Corollary}\label{prinpid}
Let $\mathbb{A}$ be a principal ideal domain, $a,b\in \mathbb{A}$ with $r={\rm gcd}(a,b)$, $a=ra'$, and $b=rb'$.
Then
\begin{description}
\item[(i)] $T_n(a,b)\sim_u {\rm diag}(r,\ldots,r,rf_n(a',b'))$ for all $n\geq 2$,

\item[(ii)]  $P_n(a,b)\sim_u {\rm diag}(r,\ldots,r,(a+2b)f_{n-1}(a',b'))$  for all $n\geq 2$,

\item[(iii)] $K_n(a,b)\sim_u {\rm diag}(r,a,\ldots,a,a'(a+nb))$  for all $n\geq 2$, and

\item[(iv)] $C_n(a,b)\sim_{u} rI_{n-2} \oplus C$ for all $n\geq 4$, where
\[
C=
\begin{cases}
f_q(a',b')\left(\begin{array}{cc}
a&2b\\
2b&a
\end{array}\right) & \text{ if } n-2=2q,\\
\big(f_{q+1}(a',b')-b'f_q(a',b')\big)\left(\begin{array}{cc}
r&0\\
0&a+2b
\end{array}\right) & \text{ if } n-2=2q+1.
\end{cases}
\]
\end{description}
\end{Corollary}
\begin{proof}
Let $X_n(a,b)$ be either one of the matrices $T_n(a,b)$, $P_n(a,b)$, $K_n(a,b)$, or $C_n(a,b)$, then $X_n(a,b)=rX_n(a',b')$.
On the other hand, since $r={\rm gcd}(a,b)$  if and only if $1={\rm gcd}(a',b')$ if and only if
the equation $a'x+b'y=1$ has solution in $\mathbb{A}$.
Then, we get the result applying theorem~\ref{prin} to $X_n(a',b')$.
\end{proof}


\section{Some applications}\label{sec2}
In this section we will apply the equivalences of the matrices $K_n(a,b)$, $C_n(a,b)$, $P_n(a,b)$ and $T_n(a,b)$ obtained in the previous section 
to the cases when $a$ and $b$ are in the ring of integers and in the subring of matrices of the form $K_n(a,b)$ where $a,b$
are in a commutative ring with identity $\mathbb{A}$.

At this point we need to introduce some definitions.
Given a simple graph $\mathcal{G}$ and a natural number $l\geq 1$, the $l$-duplication of $\mathcal{G}$, denoted by $\mathcal{G}(l)$,
is the multigraph obtained from $\mathcal{G}$ when we replace every edge of  $\mathcal{G}$ by $l$ parallel edges.
Note that $L(\mathcal{G}(l))=lL(\mathcal{G})$ for any graph $\mathcal{G}$.

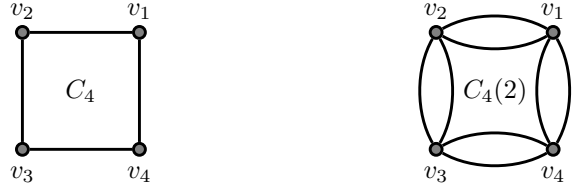
\begin{figure}[h] \centering
		\begin{tikzpicture}[line width=1.1pt, scale=1.1]
			\tikzstyle{every node}=[inner sep=0pt, minimum width=4.5pt] 
 			\draw (0,0) {
			+(45:1) node[draw, circle, fill=gray] (v1) {}
 			+(135:1) node[draw, circle, fill=gray] (v2) {}
 			+(225:1) node[draw, circle, fill=gray] (v3) {}
 			+(315:1) node[draw, circle, fill=gray] (v4) {}
 			(v1) to (v2) to (v3) to (v4) to (v1)
 			(v1)+(0,0.25) node {\small $v_1$}
 			(v2)+(0,0.25) node {\small $v_2$}
 			(v3)+(0,-0.3) node {\small $v_3$}
 			(v4)+(0,-0.3) node {\small $v_4$}
			(0,0) node {\small $C_4$}
			};
			
			 \draw (5,0) {
			+(45:1) node[draw, circle, fill=gray] (v1) {}
 			+(135:1) node[draw, circle, fill=gray] (v2) {}
 			+(225:1) node[draw, circle, fill=gray] (v3) {}
 			+(315:1) node[draw, circle, fill=gray] (v4) {}
			(v1) .. controls +(150:0.5) and +(30:0.5).. (v2)
			(v1) .. controls +(210:0.5) and +(-30:0.5).. (v2)
			(v2) .. controls +(240:0.5) and +(120:0.5).. (v3)
			(v2) .. controls +(300:0.5) and +(60:0.5).. (v3)
			(v4) .. controls +(150:0.5) and +(30:0.5).. (v3)
			(v4) .. controls +(210:0.5) and +(-30:0.5).. (v3)
			(v1) .. controls +(240:0.5) and +(120:0.5).. (v4)
			(v1) .. controls +(300:0.5) and +(60:0.5).. (v4)
 			(v1)+(0,0.25) node {\small $v_1$}
 			(v2)+(0,0.25) node {\small $v_2$}
 			(v3)+(0,-0.3) node {\small $v_3$}
 			(v4)+(0,-0.3) node {\small $v_4$}
			+(-0.7,0.7) node {\small $C_4(2)$}
			};		
		\end{tikzpicture}
	\caption{\small The simple cycle $C_4$ and its $2$-duplication.}
	\label{duplication}
\end{figure}

Given a graph $\mathcal{G}$ and a natural number $k$, the $k$-cone of $\mathcal{G}$, denoted by $c_k(\mathcal{G})$,
is the multigraph obtained from $\mathcal{G}$ by adding a new vertex $s$ and adding $k$ parallel edges between $s$ and all the vertices of $\mathcal{G}$.
Note that $L(c_k(\mathcal{G}),s)=L(\mathcal{G})+kI_{|V(\mathcal{G})|}$ for any graph $\mathcal{G}$.

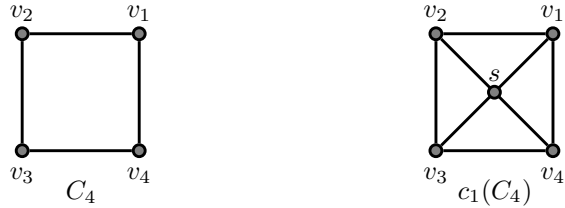
\begin{figure}[h] \centering
		\begin{tikzpicture}[line width=1.1pt, scale=1.1]
			\tikzstyle{every node}=[inner sep=0pt, minimum width=4.5pt] 
 			\draw (0,0) {
			+(45:1) node[draw, circle, fill=gray] (v1) {}
 			+(135:1) node[draw, circle, fill=gray] (v2) {}
 			+(225:1) node[draw, circle, fill=gray] (v3) {}
 			+(315:1) node[draw, circle, fill=gray] (v4) {}
 			(v1) to (v2) to (v3) to (v4) to (v1)
 			(v1)+(0,0.25) node {\small $v_1$}
 			(v2)+(0,0.25) node {\small $v_2$}
 			(v3)+(0,-0.3) node {\small $v_3$}
 			(v4)+(0,-0.3) node {\small $v_4$}
			+(-0.7,-0.5) node {\small $C_4$}
			};
			
			 \draw (5,0) {
			+(45:1) node[draw, circle, fill=gray] (v1) {}
 			+(135:1) node[draw, circle, fill=gray] (v2) {}
 			+(225:1) node[draw, circle, fill=gray] (v3) {}
 			+(315:1) node[draw, circle, fill=gray] (v4) {}
 			+(0:0) node[draw, circle, fill=gray] (s) {}
 			(v1) to (v2) to (v3) to (v4) to (v1)
			(v1) to (v3) (v2) to (v4)
 			(s)+(0,0.2) node {\small $s$}
 			(v1)+(0,0.25) node {\small $v_1$}
 			(v2)+(0,0.25) node {\small $v_2$}
 			(v3)+(0,-0.3) node {\small $v_3$}
 			(v4)+(0,-0.3) node {\small $v_4$}
			+(-0.7,-0.5) node {\small $c_1(C_4)$}
			};	
		\end{tikzpicture}
	\caption{\small The cycle $C_4$ and its $1$-cone.}
	\label{cone}
\end{figure}

The most direct application of the results obtained in the first section is when the base ring $\mathbb{A}$ is the ring of integers and the matrices 
are the reduced Laplacian matrix of the $n$-cone of the 
thick path, the thick cycle and the thick complete graph.

\newpage
\begin{Corollary}\label{CaminoGordo}
For all $m\geq 0$, $l\geq 1$, and $n\geq 2$, 
\begin{figure}[h]\centering
\begin{tikzpicture}[line width=1.1pt, scale=1]
	\draw (0,0) {
	+(0,0) node[draw, circle, fill=gray, inner sep=0pt, minimum width=4pt] (s) {}
	+(-1.8,-1.5) node[draw, circle, fill=gray,inner sep=0pt, minimum width=4pt]  (v1) {}
	+(-0.9,-1.5) node[draw, circle, fill=gray,inner sep=0pt, minimum width=4pt]  (v2) {}
	+(0.9,-1.5) node[draw, circle, fill=gray,inner sep=0pt, minimum width=4pt]  (v3) {}
	+(1.8,-1.5) node[draw, circle, fill=gray,inner sep=0pt, minimum width=4pt]  (v4) {}
	+(-0.3,-1.5) node[draw, circle, fill=black, inner sep=0pt, minimum width=1.6pt]  {}
	+(0,-1.5) node[draw, circle, fill=black, inner sep=0pt, minimum width=1.6pt]  {}
	+(0.3,-1.5) node[draw, circle, fill=black, inner sep=0pt, minimum width=1.6pt]  {}
	 (s) to (v1) (s) to (v2) (s) to (v3) (s) to (v4)
	 (v1) to (v2) (v3) to (v4)
	 (s)+(0,+0.3) node {$s$}
	 (v1)+(0,-0.3) node {$v_1$}
	 (v2)+(0,-0.3) node {$v_2$}
	 (v3)+(0,-0.3) node {$v_{n-1}$} 
	 (v4)+(0.1,-0.3) node {$v_n$}
	 (v1)+(0.65,0.8) node {$m$}
	 (v1)+(1.55,0.6) node {$m$}
	 (v1)+(2.1,0.6) node {$m$}    
	 (v1)+(2.95,0.8) node {$m$}  
	 (v1)+(0.6,0.2) node {$l$} 
	 (v3)+(0.4,0.2) node {$l$} 
	 };
\end{tikzpicture}
\end{figure}
let $c_m(\mathcal{P}_n(l))$ be the $m$-cone of the thick path with all the edges with multiplicities equal to $l$.  
Then 
\[
K(c_m(\mathcal{P}_n(l)))=\mathbb{Z}_r^{n-1} \oplus \mathbb{Z}_{mf_{n-1}(m+2l,-l)/r^{n-1}},
\]
where $r={\rm gcd}(l,m)$.
\end{Corollary}
\begin{proof}
Since the reduced Laplacian matrix of $c_m(\mathcal{P}_n(l))$, $L(c_m(\mathcal{P}_n(l)),s)$, is equal to $P_n(m+2l,-l)$
and ${\rm gcd}(m+2l,-l)={\rm gcd}(l,m)=r$, then by corollary~\ref{prinpid} $(ii)$ we get the result.
\end{proof}

\begin{Corollary}\label{CicloGordo}
For all $m\geq 0$, $l\geq 1$, and $n\geq 4$, let
\begin{figure}[h]\centering
\begin{tikzpicture}[line width=1.1pt, scale=1]
	\draw (0,0) {
	+(0,0) node[draw, circle, fill=gray, inner sep=0pt, minimum width=4pt] (s) {}
	+(60:1.4) node[draw, circle, fill=gray,inner sep=0pt, minimum width=4pt]  (v1) {}
	+(0:1.4) node[draw, circle, fill=gray,inner sep=0pt, minimum width=4pt]  (v2) {}
	+(-60:1.4) node[draw, circle, fill=gray,inner sep=0pt, minimum width=4pt]  (v3) {}
	+(-120:1.4) node[draw, circle, fill=gray,inner sep=0pt, minimum width=4pt]  (v4) {}
	+(-180:1.4) node[draw, circle, fill=gray,inner sep=0pt, minimum width=4pt]  (v5) {}
	+(-240:1.4) node[draw, circle, fill=gray,inner sep=0pt, minimum width=4pt]  (v6) {}
	+(80:1.4) node[draw, circle, fill=black, inner sep=0pt, minimum width=1.6pt]  {}
	+(90:1.4) node[draw, circle, fill=black, inner sep=0pt, minimum width=1.6pt]  {}
	+(100:1.4) node[draw, circle, fill=black, inner sep=0pt, minimum width=1.6pt]  {}
	 (s) to (v1) (s) to (v2) (s) to (v3) (s) to (v4) (s) to (v5) (s) to (v6)
	 (v1) to (v2) to (v3) to (v4) to (v5) to (v6)
	 (s)+(0,0.3) node {$s$}
	 (v1)+(0.2,0.2) node {$v_3$}
	 (v2)+(0.3,0) node {$v_2$}
	 (v3)+(0,-0.3) node {$v_{1}$} 
	 (v4)+(0.1,-0.3) node {$v_{n}$}
	 (v5)+(-0.3,-0.2) node {$v_{n-1}$}
	 (v6)+(-0.2,0.2) node {$v_{n-2}$}
	 (s)+(40:0.7) node {$m$}
	 (s)+(-15:0.7) node {$m$}
	 (s)+(-80:0.7) node {$m$}    
	 (s)+(-140:0.7) node {$m$} 
	 (s)+(-195:0.7) node {$m$}  
	 (s)+(-260:0.7) node {$m$}   
	 (s)+(30:1.4) node {$l$} 
	 (s)+(-30:1.4) node {$l$} 
	 (s)+(-90:1.4) node {$l$} 
	 (s)+(-150:1.4) node {$l$} 
	 (s)+(-210:1.4) node {$l$} 
	 };
\end{tikzpicture}
\vspace{-3mm}
\end{figure}
$c_m(\mathcal{C}_n(l))$ be the $m$-cone of the thick cycle where all the edges has multiplicity equal to $l$.
Then 
\[
K(c_m(\mathcal{C}_n(l)))=
\begin{cases}
\mathbb{Z}_r^{n-2} \oplus \mathbb{Z}_{rs_q} \oplus \mathbb{Z}_{ms_q} \text{ if } n-2=2q+1,\\
\mathbb{Z}_r^{n-2} \oplus \mathbb{Z}_{rt_q} \oplus \mathbb{Z}_{m(m+4l)t_q/r}   \text{ if } n-2=2q \text{ and } m/r \text{ is odd},\\
\mathbb{Z}_r^{n-2} \oplus \mathbb{Z}_{2rt_q} \oplus \mathbb{Z}_{m(m+4l)t_q/2r} \text{ if } n-2=2q \text{ and } m/r \text{ is even},
\end{cases}
\]
where $r={\rm gcd}(l,m)$, $s_q=(f_{q+1}(m+2l,-l)+lf_q(m+2l,-l))/r^{q+1}$, and $t_q=f_q(m+2l,-l)/r^q$.
\end{Corollary}
\begin{proof}
Since $L(c_m(\mathcal{C}_n(l)),s)$ is equal to $C_n(m+2l,-l)$ and ${\rm gcd}(m+2l,-l)={\rm gcd}(l,m)=r$, then by corollary~\ref{prinpid} $(iv)$,
$C_n(m+2l,-l)\sim_{u} rI_{n-2} \oplus C$, where
\[
C=
\begin{cases}
f_q(m+2l,-l)/r^q
\left(
\begin{array}{cc}
m+2l&-2l\\
-2l&m+2l
\end{array}
\right) & \text{ if } n-2=2q,\\
\big(f_{q+1}(m+2l,-l)+lf_q(m+2l,-l)\big)/r^{q+1}
\left(\begin{array}{cc}
r&0\\
0&m
\end{array}\right) & \text{ if } n-2=2q+1.
\end{cases}
\]
Finally, 
\[
\left(
\begin{array}{cc}
m+2l&-2l\\
-2l&m+2l
\end{array}
\right)
\sim_u
\begin{cases}
\left(
\begin{array}{cc}
r&0\\
0&(m^2+4ml)/r
\end{array}
\right)
& \text{ if } m/r \text{ is odd},\\
\\
\left(
\begin{array}{cc}
2r&0\\
0&(m^2+4ml)/2r
\end{array}
\right)
& \text{ if } m/r \text{ is even}.
\end{cases}
\]
\end{proof}

\begin{Corollary}\label{CompletaGorda}
For all $m\geq 0$, $l\geq 1$, and $n\geq 4$, let
$c_m(\mathcal{K}_n(l))$ be the $m$-cone of the thick complete graph where all the edges have multiplicity equal to $l$.
Then 
\[
K(c_m(\mathcal{K}_n(l)))=\mathbb{Z}_r \oplus \mathbb{Z}_{m+nl}^{n-2} \oplus \mathbb{Z}_{m(m+nl)/r},
\]
where $r={\rm gcd}(l,m)$.
\end{Corollary}
\begin{proof}
Since the reduced Laplacian matrix, $L(c_m(\mathcal{K}_n(l)),s)$, is equal to $K_n(m+nl,-l)$ and ${\rm gcd}(m+nl,-l)={\rm gcd}(l,m)=r$, 
then by corollary~\ref{prinpid} $(iii)$ we get the result.
\end{proof}

\subsection{The subring $K_n(\mathbb{A})$}

In this part we will turn our attention to the case when $\mathbb{A}$ is the subring of matrices given by
\[
K_n(\mathbb{A})=\{K_n(a,b)\, | \, a,b\in \mathbb{A}\}\subset M_n(\mathbb{A}).
\] 
At first, we will prove that $K_n(\mathbb{A})$ is a subalgebra of $M_n(\mathbb{A})$.

\begin{Lemma}\label{algebra}
If $a,b,c,d,\alpha\in \mathbb{A}$, then
\begin{description}
\item[(i)] $\alpha\cdot K_n(a,b)=K_n(\alpha \cdot a,\alpha \cdot b)$,
\item[(ii)] $K_n(a,b)+K_n(c,d)=K_n(a+c,b+d)$,
\item[(iii)] $K_n(a,b)\cdot K_n(c,d)=K_n(ac,ad+bc+nbd)$,
\item[(iv)] $K_n(a,b)^m=K_n(a^m,p_m(a,b))$, 
\end{description}
where the polynomials $p_{m,n}(x,y)\in \mathbb{A}[x,y]$ satisfy the recurrence relation 
\[
p_{m}^n(x,y)=(x+ny)p_{m-1}^n(x,y)+yx^{m-1}
\]
with initial value $p_{0}^n(x,y)=0$. 
\end{Lemma}
\begin{proof}
The parts $(i)$ and $(ii)$ are straightforward.
$(iii)$ Since $A(\mathcal{K}_n)^2=(n-1)I_n+(n-2)A(\mathcal{K}_n)$, 
then $K_n(a,b)\cdot K_n(c,d)=((a+b)I_n+bA(\mathcal{K}_n))\cdot((c+d)I_n+dA(\mathcal{K}_n))=(ac+ad+bc+nbd)I_n+(ad+bc+nbd)A(\mathcal{K}_n)=K_n(ac,ad+bc+nbd)$.

$(iv)$
We will use induction on $m$.
The result is clear for $m=1$ because $p_{1}^n(a,b)=(a+nb)p_{0}^n(a,b)+ba^{0}=b$.
Assume that the result is true for all the natural numbers less or equal to $m-1$.
Thus 
\begin{eqnarray*}
K_n(a,b)^m &=& K_n(a,b)^{m-1}\cdot K_n(a,b)=K_n(a^{m-1},p_{m-1}^n(a,b))\cdot K_n(a,b)\\
&=&K_n(a^{m},(a+nb)p_{m-1}^n(a,b)+ba^{m-1})=K_n(a^m,p_m^n(a,b)).
\end{eqnarray*}
\end{proof}

\begin{Remark}
Note that $K_n(\mathbb{A})$ is a commutative ring with identity because $K_n(1,0)=I_n$ for all $n\in \mathbb{N}$.
On the other hand, 
since $K_n(0,1)K_n(-n,1)=0=K_n(0,0)$, then $K_n(\mathbb{A})$ is not a principal ideal domain  because it has zero divisors.
\end{Remark}

\begin{Remark}
Using induction on $n$ is not difficult to see that
\[
p_{m}^n(x,y)=\sum_{i=1}^m n^{i-1}\binom{m}{i}x^{m-i}y^{i}.
\]
Also, since $p_m^n(n,-1)=(n-n)p_{m-1}^n(n,-1)-(n)^{m-1}=-n^{m-1}$, then 
\[
K_n(n,-1)^m=n^{m-1}K_n(n,-1) \text{ for all }m\in \mathbb{N}.
\]
\end{Remark}

Now, we will apply theorem~\ref{prin} $(iii)$ when the base ring is the subring of matrices with entries in $K_n(\mathbb{A})$
to obtain a theorem that will be a powerful tool to calculate the critical group of several graphs.

Given $A=(a_{i,j}),B=(b_{i,j})\in M_n(\mathbb{A})$ and $m\geq 2$, let
\[
\Phi_m(A,B)
=
\left(
\begin{array}{ccc}
K_m(a_{1,1},b_{1,1})&\cdots&K_m(a_{1,n},b_{1,n})\\
\vdots& \ddots& \vdots\\
K_m(a_{n,1},b_{n,1})&\cdots&K_m(a_{n,n},b_{n,n})
\end{array}
\right)
\in M_n(K_m(\mathbb{A}))
\subseteq M_{nm}(\mathbb{A}).
\]

\medskip

As the next theorem will show, if a matrix has the block structure of $\Phi_m(A,B)$, 
then we can get a simpler equivalent matrix.

\begin{Theorem}\label{CompletaMatrices}
Let $A=(a_{i,j}),B=(b_{i,j})\in M_n(\mathbb{A})$ and $n\geq 2$, then
\[
\Phi_m(A,B)
\sim_e
\left[ 
\bigoplus_{i=1}^{m-2}  A
\right]
\oplus 
\left( 
\begin{array}{cc}
A&B\\
0 &A+mB
\end{array}
\right)
\]
\end{Theorem}
\begin{proof}
Let $P_n$ and $Q_n$ be as in claim~\ref{claimK}, then
$I_m \otimes P_n, I_m\otimes Q_n\in M_{nm}(\mathbb{A})$ are elementary matrices and
\begin{eqnarray*}
(I_m \otimes P_n)\cdot \Phi_m(A,B) \cdot (I_m\otimes Q_n)
&=&
\left(
\begin{array}{ccc}
D_m(a_{1,1},b_{1,1})&\cdots&D_m(a_{1,n},b_{1,n})\\
\vdots& \ddots& \vdots\\
D_m(a_{n,1},b_{n,1})&\cdots&D_m(a_{n,n},b_{n,n})
\end{array}
\right)\\
&\sim_e&
\left[ 
\bigoplus_{i=1}^{m-2}  A
\right]
\oplus 
\left( 
\begin{array}{cc}
A&B\\
0 &A+mB
\end{array}
\right),
\end{eqnarray*}
where $D_m(a_{i,j},b_{i,j})=
\left(\begin{array}{cc}
a_{i,j}&b_{i,j}\\
-ma_{i,j}&a_{i,j}
\end{array}\right)
\oplus a_{i,j}I_{m-2}$.
\end{proof}

In the last part of this paper, we will use theorem~\ref{CompletaMatrices} to find the critical group
of some graphs whose Laplacian matrix is given by $\Phi_m(A,B)$ for $A,B\in M_n(\mathbb{A})$.
The simplest case is when $n=2$.
If $\Phi_m(A,B)$ for $A,B\in M_2(\mathbb{A})$ is the Laplacian matrix of a graph, then the vertex set of the graph can be partitioned 
in two sets and the incidence structure between these sets is given by a matrix in $K_m(\mathbb{Z})$.
In this sense, we will define the following families of graphs:

Let $U=\{u_1,u_2,\ldots,u_m\}$, $V=\{v_1,v_2,\ldots,v_m\}$, and $\mathcal{C}_{u,v}$ be the graph with $U\cup V$ as vertex set
and edge set equal to $E_u\cup E_{v}\cup E_{u,v}$, where
\[
E_u
=
\begin{cases}
\emptyset  & \text{ if } u=m,\\
u_iu_j \text{ for all } i\neq j\in \{1,2,\cdots, m\}& \text{ if } u=M,
\end{cases}
\]
similarly for $E_v$, 
and
\[
E_{u,v}
=
\begin{cases}
u_iu'_j \text{ for all } i,j\in \{1,2,\cdots, m\}& \text{ if }\mathcal{C}=\mathcal{K},\\
u_iu'_j \text{ for all } i\neq j\in \{1,2,\cdots, m\}& \text{ if }\mathcal{C}=\mathcal{L},\\
u_iu'_i \text{ for all } i\in \{1,2,\cdots, m\}& \text{ if }\mathcal{C}=\mathcal{M}.
\end{cases}
\]

Note that $\mathcal{K}_{m,m}$ is the bipartite complete graph with $m$ vertices in each partition and 
$\mathcal{L}_{m,m}$ is the bipartite complete graph with $2m$ vertices minus a matching.
The Laplacian matrix of all these graphs can be represented by $\Phi_m(A,B)$ for some  two by two matrices $A$ and $B$.

\newpage

For instance, the graph $\mathcal{M}_{M,m}$ is illustrated in figure~\ref{cone}.
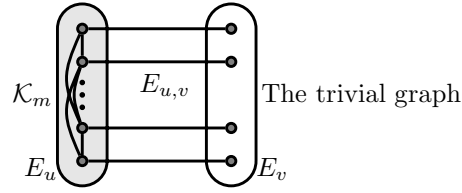
\begin{figure}[h] \centering
		\begin{tikzpicture}[line width=1.1pt, scale=1.1]
			\fill[gray!20!white] (0,-0.8) -- (0,0.8) arc (0:180:0.3) -- (-0.6,-0.8) arc (180:360:0.3);
 			\draw (0,0) {
			+(0,-0.8) -- (0,0.8) arc (0:180:0.3) -- (-0.6,-0.8) arc (180:360:0.3)
			+(-0.3,1.6) node[inner sep=0pt, minimum width=3.5pt,draw, circle, fill=gray] (v1) {}
 			+(-0.3,1.2) node[inner sep=0pt, minimum width=3.5pt,draw, circle, fill=gray] (v2) {}
 			+(-0.3,0.95) node[inner sep=0pt, minimum width=1pt,draw, circle, fill=black] () {}
	 		+(-0.3,0.8) node[inner sep=0pt, minimum width=1pt,draw, circle, fill=black] () {}
 			+(-0.3,0.65) node[inner sep=0pt, minimum width=1pt,draw, circle, fill=black] () {}
 			+(-0.3,0.4) node[inner sep=0pt, minimum width=3.5pt,draw, circle, fill=gray] (v3) {}
 			+(-0.3,0) node[inner sep=0pt, minimum width=3.5pt,draw, circle, fill=gray] (v4) {}
 			+(1.5,1.6) node[inner sep=0pt, minimum width=3.5pt,draw, circle, fill=gray] (v5) {}
 			+(1.5,1.2) node[inner sep=0pt, minimum width=3.5pt,draw, circle, fill=gray] (v6) {}
 			+(1.5,0.4) node[inner sep=0pt, minimum width=3.5pt,draw, circle, fill=gray] (v7) {}
 			+(1.5,0) node[inner sep=0pt, minimum width=3.5pt,draw, circle, fill=gray] (v8) {}
			(v2) .. controls +(252:0.4) and +(108:0.4).. (v3)
			(v1) .. controls +(246:0.6) and +(114:0.6).. (v3)
			(v2) .. controls +(246:0.6) and +(114:0.6).. (v4)
 			(v1) to (v2) (v3) to (v4) (v1) to (v5) (v2) to (v6) (v3) to (v7) (v4) to (v8) 
 			+(-2.4,0.8) node {\small $\mathcal{K}_m$}
			 +(-0.8,0.9) node {\small $E_{u,v}$}
 			+(-2.3,-0.1) node {\small $E_u$}
			+(0.3,0) -- (1.8,0.8) arc (0:180:0.3) -- (1.2,-0.8) arc (180:360:0.3)
 			+(0.2,-0.1) node {\small $E_v$}
 			+(-5,0) node {\small \mbox{}}
			};
			\draw (3.07,0) node {\small \text{The trivial graph}};
		\end{tikzpicture}
	\caption{\small The graph $\mathcal{M}_{M,m}$.}
	\label{cone}
\end{figure}

Before using theorem~\ref{CompletaMatrices} in order to calculate the critical group of the $n$-cones of graphs $\mathcal{C}_{a,b}$, 
we will introduce a theorem that gives an equivalent matrix of $\mathcal{C}_{a,b}$ when the base ring is a general commutative ring with identity.

\begin{Corollary}\label{BipartitaCompleta1}
Let $m\geq 2$, $a,b\in \mathbb{A}$ such that the equation $ax+by=1$ 
has a solution in $\mathbb{A}$, and $C_{u,u}(a,b)=aI_{2m}+bA(\mathcal{C}_{u,u})$.
Then
\begin{description}
\item[(i)]
$
K_{m,m}(a,b)
\sim_u
I_2 \oplus aI_{2(m-2)} \oplus 
a\left( 
\begin{array}{cc}
a&mb\\
mb &a
\end{array}
\right)
$,
\item[(ii)]
$
L_{m,m}(a,b)
\sim_u
I_{m} \oplus (a^2-b^2)I_{m-2} \oplus 
\left( 
\begin{array}{cc}
a^2&(m-2)ab\\
(m-2)ab &a^2-(m-1)b^2
\end{array}
\right)
$,
\item[(iii)]
$
L_{M,M}(a,b)
\sim_u
I_{m-1} \oplus a(a-2b)I_{m-2} \oplus 
\left( 
\begin{array}{ccc}
a(a-2b)&ab & 0\\
0&a+2(m-1)b& 0\\
0& (m-1)b & a
\end{array}
\right)
$,
\item[(iv)]
$
M_{M,M}(a,b)
\sim_u
I_{m+1} \oplus a(a-2b)I_{m-2} \oplus 
\left( 
\begin{array}{cc}
a(a-2b)&-b^2(2a+(m-2)b)\\
0& (a+(m-1)b)^2-b^2
\end{array}
\right)
$.
\end{description}
\end{Corollary}
\begin{proof}
$(i)$ Since $K_{m,m}(a,b)=\Phi_m(A,B)$ for 
$
A
=
\left( 
\begin{array}{cc}
a&0\\
0 & a
\end{array}
\right)
$
and
$
B
=
\left( 
\begin{array}{cc}
0&b\\
b &0
\end{array}
\right)
$,
then by theorem~\ref{CompletaMatrices}
\[
K_{m,m}(a,b)
\sim_e
\left[ 
\bigoplus_{i=1}^{m-2}  
\left( 
\begin{array}{cc}
a&0\\
0 &a
\end{array}
\right)
\right]
\oplus 
\left( 
\begin{array}{cccc}
a& 0 & 0&b\\
0 &a &b&0\\
0& 0& a&mb\\
0& 0& mb&a
\end{array}
\right).
\]
Moreover, since
\[
\left( 
\begin{array}{cccc}
a& 0 & 0&b\\
0 &a &b&0\\
0& 0& a&mb\\
0& 0& mb&a
\end{array}
\right)
\left( 
\begin{array}{cccc}
x& 0 & 0&-b\\
0 &1 &0&0\\
0& 0& 1&0\\
y& 0& 0&a
\end{array}
\right)
=
\left( 
\begin{array}{cccc}
1& 0 & 0&0\\
0 &a &b&0\\
mby& 0& a&mab\\
ay& 0& mb&a^2
\end{array}
\right)
\sim_e
\left( 
\begin{array}{cccc}
1& 0 & 0&0\\
0 &a &b&0\\
0& 0& a&mab\\
0& 0& mb&a^2
\end{array}
\right)
\]
and
\[
\left( 
\begin{array}{ccc}
a &b&0\\
0& a&mab\\
0& mb&a^2
\end{array}
\right)
\left( 
\begin{array}{ccc}
x &-b&0\\
y& a&0\\
0& 0&1
\end{array}
\right)
=
\left( 
\begin{array}{ccc}
1 &0&0\\
ay& a^2&mab\\
mby& mab&a^2
\end{array}
\right)
\sim_e
\left( 
\begin{array}{ccc}
1 &0&0\\
0& a^2&mab\\
0& mab&a^2
\end{array}
\right),
\]
then
\[
\left( 
\begin{array}{cccc}
a& 0 & 0&b\\
0 &a &b&0\\
0& 0& a&mb\\
0& 0& mb&a
\end{array}
\right)
\sim_u
I_2
\oplus
a\left( 
\begin{array}{cc}
a&mb\\
mb &a
\end{array}
\right).
\]
Hence
\[
K_{m,m}(a,b)
\sim_u
I_2 \oplus aI_{2(m-2)} \oplus 
a\left( 
\begin{array}{cc}
a&mb\\
mb &a
\end{array}
\right).
\]

$(ii)$ Since
$L_{m,m}(a,b)=\Phi_m(A,B)$ 
for 
$
A
=
\left( 
\begin{array}{cc}
a&-b\\
-b & a
\end{array}
\right)
$
and
$
B
=
\left( 
\begin{array}{cc}
0&b\\
b &0
\end{array}
\right)
$,
then by theorem~\ref{CompletaMatrices}
\[
L_{m,m}(a,b)
\sim_e
\left[ 
\bigoplus_{i=1}^{m-2}  
\left( 
\begin{array}{cc}
a&-b\\
-b &a
\end{array}
\right)
\right]
\oplus 
\left( 
\begin{array}{cccc}
a& -b & 0&b\\
-b &a &b&0\\
0& 0& a&(m-1)b\\
0& 0& (m-1)b&a
\end{array}
\right).
\]
Moreover, since
\[
\left( 
\begin{array}{cc}
a& -b\\
-b& a
\end{array}
\right)
\left( 
\begin{array}{cc}
x& b\\
-y& a
\end{array}
\right)
=\left( 
\begin{array}{cc}
1& 0\\
-by-ay& a^2-b^2
\end{array}
\right)
\sim_e
\left( 
\begin{array}{cc}
1& 0\\
0& a^2-b^2
\end{array}
\right),
\]
{\small
\[
\left( 
\begin{array}{cccc}
a& -b & 0&b\\
-b &a &b&0\\
0& 0& a&(m-1)b\\
0& 0& (m-1)b&a
\end{array}
\right)
\sim_e
\left( 
\begin{array}{cccc}
a& 0 & 0&b\\
-b &a &b&0\\
0& (m-1)b& a&(m-1)b\\
0& a& (m-1)b&a
\end{array}
\right),
\]
}
{\footnotesize
\begin{eqnarray*}
\left( 
\begin{array}{cccc}
1& 0 & 0&0\\
* &a &b&b^2\\
*& (m-1)b& a&(m-1)ab\\
*& a& (m-1)b&a^2
\end{array}
\right)
&=&
\left( 
\begin{array}{cccc}
a& 0 & 0&b\\
-b &a &b&0\\
0& (m-1)b& a&(m-1)b\\
0& a& (m-1)b&a
\end{array}
\right)
\left( 
\begin{array}{cccc}
x& 0 & 0&-b\\
0 &1 &0&0\\
0& 0& 1&0\\
y& 0& 0&a
\end{array}
\right)\\
&\sim_e&
I_1\oplus
\left( 
\begin{array}{ccc}
 a & b&0\\
(m-1)b &a&(m-2)ab\\
a& (m-1)b&a^2-(m-1)b^2
\end{array}
\right),
\end{eqnarray*}
}
and
{\small
\begin{eqnarray*}
\left( 
\begin{array}{ccc}
 1 & 0&0\\
* &a^2&(m-2)ab\\
*& (m-2)ab&a^2-(m-1)b^2
\end{array}
\right)
&=&
\left( 
\begin{array}{ccc}
 a & b&0\\
(m-1)b &a&(m-2)ab\\
a& (m-1)b&a^2-(m-1)b^2
\end{array}
\right)
\left( 
\begin{array}{ccc}
 x & -b&0\\
y &a&0\\
0& 0&1
\end{array}
\right)\\
&\sim_e&
I_1\oplus
\left( 
\begin{array}{cc}
a^2&(m-2)ab\\
(m-2)ab&a^2-(m-1)b^2
\end{array}
\right),
\end{eqnarray*}
}
then
\[
L_{m,m}(a,b)
\sim_u
I_{m} \oplus (a^2-b^2)I_{m-2} \oplus 
\left( 
\begin{array}{cc}
a^2&(m-2)ab\\
(m-2)ab &a^2-(m-1)b^2
\end{array}
\right).
\]
$(iii)$ Since
$L_{M,M}(a,b)=\Phi_m(A,B)$ 
for 
$
A
=
\left( 
\begin{array}{cc}
a-b&-b\\
-b & a-b
\end{array}
\right)
$
and
$
B
=
\left( 
\begin{array}{cc}
b&b\\
b &b
\end{array}
\right)
$,
then by theorem~\ref{CompletaMatrices}
\[
L_{M,M}(a,b)
\sim_e
\left[ 
\bigoplus_{i=1}^{m-2}  
\left( 
\begin{array}{cc}
a-b&-b\\
-b &a-b
\end{array}
\right)
\right]
\oplus 
\left( 
\begin{array}{cccc}
a-b& -b & b&b\\
-b &a-b &b&b\\
0& 0& a+(m-1)b&(m-1)b\\
0& 0& (m-1)b&a+(m-1)b
\end{array}
\right).
\]
Moreover, since
\[
\left( 
\begin{array}{cc}
x& -(x+y)\\
b& a-b
\end{array}
\right)
\left( 
\begin{array}{cc}
a-b& -b\\
-b& a-b
\end{array}
\right)
=
\left( 
\begin{array}{cc}
1& *\\
0& a(a-2b)
\end{array}
\right)
\sim_e
\left( 
\begin{array}{cc}
1& 0\\
0& a(a-2b)
\end{array}
\right),
\]
{\tiny
\begin{eqnarray*}
\left( 
\begin{array}{ccc}
x& -(x+y) & 0\\
b &a-b &0\\
0& 0& I_2
\end{array}
\right)
\left( 
\begin{array}{cccc}
a-b& -b & b&b\\
-b &a-b &b&b\\
0& 0& a+(m-1)b&(m-1)b\\
0& 0& (m-1)b&a+(m-1)b
\end{array}
\right)
=
\left( 
\begin{array}{cccc}
1&* & *&*\\
0&a(a-2b)&ab & ab\\
0&0&a+(m-1)b& (m-1)b\\
0&0& (m-1)b & a+(m-1)b
\end{array}
\right),
\end{eqnarray*}
}
and
\[
\left( 
\begin{array}{ccc}
a(a-2b)&ab & ab\\
0&a+(m-1)b& (m-1)b\\
0& (m-1)b & a+(m-1)b
\end{array}
\right)
\sim_e
\left( 
\begin{array}{ccc}
a(a-2b)&ab & 0\\
0&a+2(m-1)b& 0\\
0& (m-1)b & a
\end{array}
\right)
\]
then the result is followed.


$(iv)$ Since
$M_{M,M}(a,b)=\Phi_m(A,B)$ 
for 
$
A
=
\left( 
\begin{array}{cc}
a-b&b\\
b & a-b
\end{array}
\right)
$
and
$
B
=
\left( 
\begin{array}{cc}
b&0\\
0 &b
\end{array}
\right)
$,
then by theorem~\ref{CompletaMatrices}
\[
M_{M,M}(a,b)
\sim_e
\left[ 
\bigoplus_{i=1}^{m-2}  
\left( 
\begin{array}{cc}
a-b&b\\
b &a-b
\end{array}
\right)
\right]
\oplus 
\left( 
\begin{array}{cccc}
a-b& b & b&0\\
b &a-b &0&b\\
0& 0& a+(m-1)b&b\\
0& 0& b&a+(m-1)b
\end{array}
\right).
\]
Moreover, since 
{\tiny
\[
\left( 
\begin{array}{ccc}
x& x+y & 0\\
-b &a-b &0\\
0& 0& I_2
\end{array}
\right)
\left( 
\begin{array}{cccc}
a-b& b & b&0\\
b &a-b &0&b\\
0& 0& a+(m-1)b&b\\
0& 0& b&a+(m-1)b
\end{array}
\right)
=
\left( 
\begin{array}{cccc}
1& * & *&*\\
0 &a(a-2b) &-b^2&(a-b)b\\
0& 0& a+(m-1)b&b\\
0& 0& b&a+(m-1)b
\end{array}
\right)
\]
}
and
{\tiny
\[
\left( 
\begin{array}{ccc}
a(a-2b)&-b^2 & (a-b)b\\
0&a+(m-1)b& b\\
0& b & a+(m-1)b
\end{array}
\right)
\left( 
\begin{array}{ccc}
1&0 & 0\\
0&a+(m-1)b& y-(m-1)x\\
0& -b & x
\end{array}
\right)
=
\left( 
\begin{array}{ccc}
a(a-2b)&-b^2(2a+(m-2)b) & *\\
0& (a+(m-1)b)^2-b^2& *\\
0& 0 & 1
\end{array}
\right),
\]
}
then the result is followed.
\end{proof}

\begin{Remark}
Note that $\mathcal{K}_{M,M}$ is the complete graph with $2m$ vertices and $\mathcal{M}_{m,m}$ is the disjoint union of $m$ copies of $\mathcal{K}_2$.
\end{Remark}

When a graph $\mathcal{G}$ is not regular, then there is not  a straightforward way to define their matrix $G(a,b)$.
Thus, we will define $K_{m,M}(a,b)$ as $a(I_{m}\oplus 2I_m)+bA(\mathcal{K}_{m,M})$, 
$L_{m,M}(a,b)$ as $a(I_{m}\oplus 2I_m)+bA(\mathcal{L}_{m,M})$,
and $M_{m,M}(a,b)$ as $a(I_{m}\oplus (m+1)I_m)+bA(\mathcal{M}_{m,M})$.
Now, we have the following equivalent matrices of $K_{m,M}(a,b)$, $L_{m,M}(a,b)$, and $M_{m,M}(a,b)$.

\begin{Corollary}\label{BipartitaCompleta}
Let $m\geq 2$, $a,b\in \mathbb{A}$ such that the equation $ax+by=1$ has a solution in $\mathbb{A}$,
then
\begin{description}
\item[(i)]
$
K_{m,M}(a,b)
\sim_u
I_{2} \oplus aI_{m-2} \oplus 2aI_{m-2} \oplus 
a\left( 
\begin{array}{cc}
2a&-(a-mb)b\\
2m&2a
\end{array}
\right)
$,
\item[(ii)]
$
M_{m,M}(a,b)
\sim_u
I_{m} \oplus ((m+1)a^2-b^2)I_{m-2} \oplus 
\left( 
\begin{array}{cc}
(m+1)a^2-b^2&b\\
0&(a+b)(b-(m+1)a)
\end{array}
\right)
$,
\item[(iii)]
$
L_{m,M}(a,b)
\sim_u
I_{m-1} \oplus (2a^2-b^2)I_{m-2}\oplus 
\left( 
\begin{array}{ccc}
2a^2-b^2&0& ab+mb^2\\
0&a&(m-1)b\\
0&(m-1)b& 2a+mb
\end{array}
\right)
$.
\end{description}
\end{Corollary}
\begin{proof}
$(i)$
Since $K_{m,M}(a,b)=\Phi_m(A,B)$ for 
$
A
=
\left( 
\begin{array}{cc}
a&0\\
0 & 2a
\end{array}
\right)
$
and
$
B
=
\left( 
\begin{array}{cc}
0&b\\
b &b
\end{array}
\right)
$,
then by theorem~\ref{CompletaMatrices}
\[
K_{m,M}(a,b)\sim_e 
\left[\bigoplus^{m-2} \left(\begin{array}{cccc}a&0\\0&2a\end{array}\right)\right]
\oplus 
\left(
\begin{array}{cccc}
a&0&0&b\\
0&2a&b&b\\
0&0&a&mb\\
0&0&mb&2a+mb
\end{array}
\right).
\]
Moreover,
\begin{eqnarray*}
{\footnotesize
\left(
\begin{array}{cccc}
a&0&0&b\\
0&2a&b&b\\
0&0&a&mb\\
0&0&mb&2a+mb
\end{array}
\right)
\left(
\begin{array}{cccc}
x&0&0&-b\\
0&1&0&0\\
0&0&1&0\\
y&0&0&a
\end{array}
\right)
=
\left(
\begin{array}{cccc}
1&0&0&0\\ 
*&2a&b&ab\\ 
*&0&a&mab\\ 
*&0&mb&2a^2+mab
\end{array}
\right)
\sim_e
\left(
\begin{array}{cccc}
1&0&0&0\\
0&b&2a&ab\\
0&a&0&mab\\
0&0&-2ma&2a^2
\end{array}
\right)
}
\end{eqnarray*}
and
\[
{\small
\left(
\begin{array}{ccc}
y&x&0\\
-a&b&0\\
0&0&1
\end{array}
\right)
\left(
\begin{array}{ccc}
b&2a&ab\\
a&0&mab\\
0&-2ma&2a^2
\end{array}
\right)
=
\left(
\begin{array}{ccc}
1&*&*\\
0&-2a^2&mab^2-a^2b\\
0&-2ma&2a^2
\end{array}
\right)
\sim_e
I_1
\oplus
\left(
\begin{array}{cc}
2a^2&ab(mb-a)\\
2ma&2a^2
\end{array}
\right).
}
\]
$(ii)$
Since $M_{m,M}(a,b)=\Phi_m(A,B)$ for 
$
A
=
\left( 
\begin{array}{cc}
a&b\\
b & (m+1)a
\end{array}
\right)
$
and
$
B
=
\left( 
\begin{array}{cc}
0&0\\
0 &b
\end{array}
\right)
$,
then by theorem~\ref{CompletaMatrices}
\[
M_{m,M}(a,b)
\sim 
\left[\bigoplus^{m-2} \left(\begin{array}{cc}a&b\\b&(m+1)a\end{array}\right)\right]
\oplus 
\left(
\begin{array}{cccc}
a&b&0&0\\
b&(m+1)a&0&b\\
0&0&a&b\\
0&0&b&(m+1)a+mb
\end{array}
\right)
\]
Moreover, 
\[
\left(
\begin{array}{cc}
a&b\\
b&(m+1)a
\end{array}
\right)
\left(
\begin{array}{cc}
x&-b\\
y&a
\end{array}
\right)
=
\left(
\begin{array}{cc}
1&0\\
*&(m+1)a^2-b^2
\end{array}
\right)
\sim_e 
\left(
\begin{array}{cc}
1&0\\
0&(m+1)a^2-b^2
\end{array}
\right)
\]
and
\begin{eqnarray*}
\left(
\begin{array}{cccc}
a&b&0&0\\
b&(m+1)a&0&b\\
0&0&a&b\\
0&0&b&(m+1)a+mb
\end{array}
\right)
&\sim_u&
\left(
\begin{array}{cccc}
1&0&0&0\\ 
*&(m+1)a^2-b^2&0&b\\ 
*&0&a&b\\ 
*&0&b&(m+1)a+mb
\end{array}
\right)\\
&\sim_e&
I_1 
\oplus 
\left(
\begin{array}{ccc}
b&(m+1)a+mb&0\\
a&b&0\\
0&b&(m+1)a^2-b^2
\end{array}
\right)\\
&\sim_u&
I_2 \oplus 
\left(
\begin{array}{cc}
(a+b)(b-(m+1)a)&0\\
b&(m+1)a^2-b^2
\end{array}
\right)
\end{eqnarray*}

$(iii)$
Since $L_{m,M}(a,b)=\Phi_m(A,B)$ for 
$
A
=
\left( 
\begin{array}{cc}
a&-b\\
-b & 2a
\end{array}
\right)
$
and
$
B
=
\left( 
\begin{array}{cc}
0&b\\
b &b
\end{array}
\right)
$,
then by theorem~\ref{CompletaMatrices}
\[
L_{m,M}(a,b)
\sim 
\left[\bigoplus^{m-2} \left(\begin{array}{cc}a&-b\\-b&2a\end{array}\right)\right]
\oplus 
\left(\begin{array}{cccc}a&-b&0&b\\-b&2a&b&b\\0&0&a&(m-1)b \\0&0&(m-1)b&2a+mb \end{array}\right).
\]
Moreover,
\[
\left(
\begin{array}{cc}
x&-y\\
b&a
\end{array}
\right)
\left(\begin{array}{cc}a&-b\\-b&2a\end{array}\right)
=
\left(\begin{array}{cc}1&*\\ 0&2a^2-b^2\end{array}\right)
\sim_e 
\left(\begin{array}{cc}1&0\\0&2a^2-b^2\end{array}\right),
\]
\[
\left(\begin{array}{cccc}a&-b&0&b\\-b&2a&b&b\\0&0&a&(m-1)b \\0&0&(m-1)b&2a+mb \end{array}\right)
\sim_u
\left(\begin{array}{cccc}1&0&0&0\\0&2a^2-b^2&ab&ab+b^2\\0&0&a&(m-1)b \\0&0&(m-1)b&2a+mb \end{array}\right),
\]
and 
\begin{eqnarray*}
\left(\begin{array}{ccc}2a^2-b^2&ab&ab+b^2\\0&a&(m-1)b\\0&(m-1)b&2a+mb \end{array}\right)
&\sim_e&
\left( 
\begin{array}{ccc}
2a^2-b^2&0& ab+mb^2\\
0&a&(m-1)b\\
0&(m-1)b& 2a+mb
\end{array}
\right).
\end{eqnarray*}
\end{proof}

The next corollaries calculate the critical group of the $n$-cone of the graphs $\mathcal{C}_{a,b}$.

\begin{Corollary}\label{BipartitaCompletaG1}
Let $m\geq 2$, $l\geq 1$, and $n\geq 0$, then
\[
K(c_n(\mathcal{K}_{m,m}(l)))=\mathbb{Z}_r^{2} \oplus \mathbb{Z}_{n+ml}^{2(m-2)} \oplus \mathbb{Z}_{(n+ml)s/r} \oplus \mathbb{Z}_{n(n+ml)(n+2ml)/rs},
\]
where $r={\rm gcd}(l,n)$ and $s={\rm gcd}(ml,n)$.
\end{Corollary}
\begin{proof}
Since $L(c_n(\mathcal{K}_{m,m}(l)),s)=K_{m,m}(n+ml,-l)$ and
$K_{m,m}(n+ml,-l)=rK_{m,m}((n+ml)/r,-l/r)$ with ${\rm gcd}((n+ml)/r,-l/r)=1$, 
then applying corollary~\ref{BipartitaCompleta} $(i)$ to $K_{m,m}((n+ml)/r,-l/r)$ we get that
\[
K_{m,m}(n+ml,-l)
\sim_u
rI_2 \oplus (n+ml)I_{2(m-2)} \oplus 
((n+ml)/r)\left( 
\begin{array}{cc}
n+ml&-ml\\
-ml &n+ml
\end{array}
\right).
\]
On the other hand, 
\[
\left( 
\begin{array}{cc}
n+ml&-ml\\
-ml &n+ml
\end{array}
\right)
\sim_u
\left( 
\begin{array}{cc}
s&0\\
0 &n(n+2ml)/s
\end{array}
\right),
\]
where $s={\rm gcd}(ml,n)$ and we get the result.
\end{proof}

\begin{Remark}
Note that $\mathcal{K}_{m,m}$ is the complete bipartite graph with $m$ vertices in each partition. 
Lorenzini in ~\cite{lorenzini91} calculated that
\[
K(\mathcal{K}_{m,m})=\mathbb{Z}_{m}^{2(m-2)}\oplus \mathbb{Z}_{m^2},
\] 
which agrees with the corollary~\ref{BipartitaCompletaG1} for $l=1$ and $n=0$.
Also note that $K(c_n(\mathcal{K}_{m,m}))$ has $2m-2$ invariant factors different to $1$.
\end{Remark}

\begin{Corollary}\label{BipartitaCompletaG2}
Let $m\geq 3$, $l\geq 1$, and $n\geq 0$, then
\[
K(c_n(\mathcal{L}_{m,m}(l)))=\mathbb{Z}_r^{m} \oplus \mathbb{Z}_{(s^2-l^2)/r}^{m-2} \oplus \mathbb{Z}_{rt} \oplus \mathbb{Z}_{u/r^3t},
\]
where $r={\rm gcd}(l,n)$, $s=n+(m-1)l$, $t={\rm gcd}(m-1,n)/{\rm gcd}(l,m-1,n)$, and $u=s^2(n^2+2n(m-1)l+(m-2)l^2)$.
\end{Corollary}
\begin{proof}
In a similar way that in corollary~\ref{BipartitaCompletaG1}, $L(c_n(\mathcal{L}_{m,m}(l)),s)=L_{m,m}(n+(m-1)l,-l)$ and
applying corollary~\ref{BipartitaCompleta} $(ii)$ to $L_{m,m}(a/r,b/r)$ with $a=n+(m-1)l$ and $b=-l$
\[
L_{m,m}(n+(m-1)l,-l)
\sim_u
rI_m \oplus (s^2-l^2)/rI_{m-2} \oplus 
\left( 
\begin{array}{cc}
a^2/r&(m-2)ab/r\\
(m-2)ab/r &(a^2-(m-1)b^2)/r
\end{array}
\right),
\]
where $s=n+(m-1)l$.

On the other hand, it is not difficult to see that
\[
\left( 
\begin{array}{cc}
a^2/r&(m-2)ab/r\\
(m-2)ab/r &(a^2-(m-1)b^2)/r
\end{array}
\right)
\sim_u
\left( 
\begin{array}{cc}
rt&0\\
0 &u/r^3t
\end{array}
\right),
\]
where $t={\rm gcd}(m-1,n)/{\rm gcd}(l,m-1,n)$ and $u=s^2(n^2+2n(m-1)l+(m-2)l^2)$.
\end{proof}


\begin{Corollary}\label{BipartitaCompletaG3}
Let $m\geq 2$, $l\geq 1$, and $n\geq 0$, then
\[
K(c_n(\mathcal{L}_{M,M}(l)))=\mathbb{Z}_{r}^{m-1}\oplus\mathbb{Z}_{st/r}^{m-2}\oplus \mathbb{Z}_{u}\oplus \mathbb{Z}_{sv/u}\oplus \mathbb{Z}_{nst/rv},
\]
where $r={\rm gcd}(l,n)$, $s=n+2(m-1)l$, $t=n+2ml$, 
$u={\rm gcd}(n,(m-1)l)$, and $v={\rm gcd}(n,2(m-1)l^2/r)$. 
\end{Corollary}
\begin{proof}
Since $L(c_n(\mathcal{L}_{M,M}(l)),s)=L_{M,M}(n+2(m-1)l,-l)$ and $r={\rm gcd}(n+2(m-1)l,-l)$,
then applying corollary~\ref{BipartitaCompleta} $(iii)$ to $L_{M,M}(a/r,b/r)$ with $a=n+2(m-1)l$ and $b=-l$ we get that
\[
L_{M,M}(n+2(m-1)l,-l)
\sim_u
rI_{m-1} \oplus st/rI_{m-2} \oplus 
\left( 
\begin{array}{ccc}
st/r&-sl/r & 0\\
0&n& 0\\
0& -(m-1)l & s
\end{array}
\right),
\]
where $r={\rm gcd}(l,n)$, $s=n+2(m-1)l$, and $t=n+2ml$.

On the other hand, it is not difficult to see that
\[
\left( 
\begin{array}{ccc}
a(a-2b)/r&ab/r & 0\\
0&a+2(m-1)b& 0\\
0& (m-1)b & a
\end{array}
\right)
=
\left( 
\begin{array}{ccc}
st/r&-sl/r & 0\\
0&n& 0\\
0& -(m-1)l & s
\end{array}
\right)
\sim_u
u \oplus sv/u \oplus nst/rv,
\]
where $u={\rm gcd}(n,(m-1)l)$ and $v={\rm gcd}(n,2(m-1)l^2/r)$.
\end{proof}


\begin{Corollary}\label{BipartitaCompletaG4}
Let $m\geq 2$, $l\geq 1$, and $n\geq 0$, then
\[
K(c_n(\mathcal{M}_{M,M}(l)))=\mathbb{Z}_{r}^{m+1}\oplus\mathbb{Z}_{(n+ml)(n+(m+2)l)/r}^{m-2}\oplus \mathbb{Z}_{u}\oplus \mathbb{Z}_{n(n+2l)(n+ml)(n+(m+2)l)/ur^2},
\]
where $r={\rm gcd}(l,n)$, 
$u={\rm gcd}(n(n+2l),lv(n+t))/r$, and $v={\rm gcd}(m,l/r)$. 
\end{Corollary}
\begin{proof}
Since $L(c_n(\mathcal{M}_{M,M}(l)),s)=M_{M,M}(n+ml,-l)$ and $r={\rm gcd}(n+ml,-l)$,
then applying corollary~\ref{BipartitaCompleta} $(iii)$ to $M_{M,M}(a/r,b/r)$ with $a=n+ml$ and $b=-l$ we get that
\[
M_{M,M}(n+ml,-l)
\sim_u
rI_{m+1} \oplus st/rI_{m-2} \oplus 
\left( 
\begin{array}{cc}
st/r&-l^2(n+t)/r^2\\
0& n(n+2l)/r,
\end{array}
\right)
\]
where $r={\rm gcd}(l,n)$, $s=n+ml$, and $t=n+(m+2)l$.

On the other hand, it is not difficult to see that
\[
\left( 
\begin{array}{cc}
st/r&-l^2(n+t)/r^2\\
0& n(n+2l)/r
\end{array}
\right)
\sim_u
u \oplus stn(n+2l)/ur^2,
\]
where $u={\rm gcd}(n(n+2l),lv(n+t))/r$ and $v={\rm gcd}(m,l/r)$.
\end{proof}

\begin{Remark}
Note that $\mathcal{M}_{M,M}(l)$ is the cartesian product of $\mathcal{K}_2(l)$ and $\mathcal{K}_n(l)$.
A deep analysis of the cartesian product of matrices can be found in~\cite{Matrix2}. 
\end{Remark}

\begin{Corollary}\label{BipartitaCompletaG5}
Let $m\geq 2$, $l\geq 1$, and $n\geq 0$, then
\[
K(c_n(\mathcal{K}_{m,M}(l)))= \mathbb{Z}_{ml+n}^{m-2} \oplus \mathbb{Z}_{2ml+n}^{m-2} 
\oplus \mathbb{Z}_{r}^2 \oplus \mathbb{Z}_{s(n+2ml)/r^2} \oplus \mathbb{Z}_{n(n+ml)(n+2ml)/s},
\]
where $r={\rm gcd}(l,n)$ and $s={\rm gcd}(n^2,mlr)$.
\end{Corollary}
\begin{proof}
Since $L(c_n(\mathcal{K}_{m,M}(l)))= \Phi_m(A,B)$ for
$A=\left(\begin{array}{cc}ml+n&0\\0&2ml+n\end{array}\right)$ and $B=\left(\begin{array}{cc}0&-l\\-l&-l\end{array}\right)$,
then the result is followed by theorem~\ref{BipartitaCompleta} $(iv)$ because
\[
\left(
\begin{array}{cccc}
ml+n&0&0&-l\\
0&2ml+n&-l&-l\\
0&0&ml+n&-ml\\
0&0&-ml&ml+n
\end{array}
\right)
\sim_u
rI_2 \oplus (n+2ml)s/r^2 \oplus n(n+ml)(n+2ml)/s,
\]
where $r={\rm gcd}(l,n)$ and $s={\rm gcd}(n^2,mlr)$.
\end{proof}

\begin{Remark}
Note that $\mathcal{K}_{m,M}$ is the graph $\mathcal{K}_{2m}\setminus \mathcal{K}_{m}$.
In general the expression for $K(c_n(\mathcal{K}_{m,M}(l)))$ given in corollary~\ref{BipartitaCompletaG5} does not give us the invariant factors of $K(c_n(\mathcal{K}_{m,M}(l)))$.
Also note that $K(c_n(\mathcal{K}_{m,M}))$ has $2m-2$ invariant factors different to $1$.
\end{Remark}

\begin{Corollary}\label{BipartitaCompletaG6}
Let $m\geq 2$, $l\geq 1$, and $n\geq 0$, then
\[
K(c_n(\mathcal{M}_{m,M}(l)))= \mathbb{Z}_{r}^{m+1} \oplus \mathbb{Z}_{s/r}^{m-2} \oplus \mathbb{Z}_{n(n+2l)s/r^3},
\]
where $r={\rm gcd}(l,n)$ and $s=n^2+ml^2+nl(m+2)$.
\end{Corollary}
\begin{proof}
Since $L(c_n(\mathcal{M}_{m,M}(l)))= \Phi_m(A,B)$ for
$A=\left(\begin{array}{cc}l+n&-l\\-l&(m+1)l+n\end{array}\right)$ and $B=\left(\begin{array}{cc}0&0\\0&-l\end{array}\right)$,
then the result is followed by theorem~\ref{BipartitaCompleta} $(v)$ because
\[
\left(
\begin{array}{cc}
l+n&-l\\
-l&(m+1)l+n
\end{array}
\right)
\sim_u 
r \oplus s/r
\]
and
\[
\left(
\begin{array}{cccc}
l+n&-l&0&0\\
-l&(m+1)l+n&0&-l\\
0&0&l+n&-l\\
0&0&-l&l+n
\end{array}
\right)
\sim_u
rI_3 \oplus n(n+2l)s/r^3,
\]
where $r={\rm gcd}(l,n)$ and $s=n^2+ml^2+nl(m+2)$.
\end{proof}

We conclude the article with the critical group of the graph $\mathcal{L}_{m,M}$.

\begin{Corollary}\label{BipartitaCompletaG7}
Let $m\geq 2$, $l\geq 1$, and $n\geq 0$, then
\[
K(\mathcal{L}_{m,M})=\mathbb{Z}_{r}^{m} \oplus \mathbb{Z}_{s/r}^{m-2} \oplus \mathbb{Z}_{t} \oplus \mathbb{Z}_{n(n+2(m-1)l)s/tr^2},
\]
where $r={\rm gcd}(l,n)$, $s=n^2+(3m-2)nl+m(2m-3)l^2$, and $t={\rm gcd}(n,l^3(m-1)(2m-3)/r^2)$.
\end{Corollary}
\begin{proof}
Since
$L(c_n(\mathcal{L}_{m,M}(l)))= \Phi_m(A,B)$ for $A=\left(\begin{array}{cc}(m-1)l+n&l\\l&(2m-1)l+n\end{array}\right)$ and $B=\left(\begin{array}{cc}0&-l\\-l&-l\end{array}\right)$,
then the result is followed by theorem~\ref{BipartitaCompleta} $(vi)$ because
\[
\left(
\begin{array}{cc}
(m-1)l+n&l\\
l&(2m-1)l+n
\end{array}
\right)
\sim 
r\oplus s/r,
\]
where $r={\rm gcd}(l,n)$, $s=n^2+(3m-2)nl+m(2m-3)l^2$, and
\[
\left(
\begin{array}{cccc}
(m-1)l+n&l&0&-l\\
l&(2m-1)l+n&-l&-l\\
0&0&(m-1)l+n&-(m-1)l\\
0&0&-(m-1)l&(m-1)l+n
\end{array}
\right)
\sim_u
rI_2 \oplus t \oplus n(n+2(m-1)l)s/tr^2,
\]
where $t={\rm gcd}(n,l^3(m-1)(2m-3)/r^2)$.
\end{proof}



\end{document}